\documentclass[11pt,twoside]{article}
\usepackage{amsmath}
\usepackage{amsfonts}
\usepackage{amsthm,amssymb}
\usepackage{graphicx, color}

\newtheorem{defi}{Definition}[section]
\newtheorem{theorem}[defi]{Theorem}
\newtheorem{lemma}[defi]{Lemma}

\newtheorem{proposition}[defi]{Proposition}

\begin{document}
\vskip 1cm
\begin{center}
\vskip 0.1cm
{\bf\huge Existence of bounded variation solutions for a 1-Laplacian problem with vanishing potentials}
\end{center}
\vskip 0.3cm
\begin{center}

{\sc Giovany M. Figueiredo$^1$ and Marcos T. O. Pimenta$^2$,}
\\
\vspace{0.5cm}

1. Universidade Federal do Par\'a, Faculdade de Matem\'{a}tica-ICEN,\\
66075-110 - Bel\'em - PA , Brazil, \\

2. Departamento de Matem\'atica e Computa\c{c}\~ao\\ Fac. de Ci\^encias e Tecnologia, Universidade Estadual Paulista - UNESP \\
19060-900 - Presidente Prudente - SP, Brazil, \\

\medskip

e-mail addresses: giovany@ufpa.br, pimenta@fct.unesp.br (corresponding author),
 \end{center}
\vskip 1cm

\begin{abstract}

In this work it is studied a quasilinear elliptic problem in the whole space $\mathbb{R}^N$ involving the $1-$Laplacian operator, with potentials which can vanish at infinity. The Euler-Lagrange functional is defined in a space whose definition resembles $BV(\mathbb{R}^N)$ and, in order to avoid working with extensions of it to some Lebesgue space, we state and prove a version of the Mountain Pass Theorem without the Palais-Smale condition to Lipschitz continuous functionals.

\end{abstract}

\vskip 1cm

 \vskip 0.4cm

\section{Introduction and some abstract results}

\hspace{.5cm} 

In general, whenever dealing with semilinear or quasilinear elliptic problems in $\mathbb{R}^N$, it is explored the reflexivity of the Sobolev spaces $W^{m,p}(\mathbb{R}^N)$, for $1 < p < +\infty$. In fact, the weak limits of sequences, which can be minimizing, Palais-Smale, and so on, are the candidates to be weak solutions of the problems.

When dealing with problems which are modeled in the space of functions of bounded variation, $BV(\mathbb{R}^N)$, the situation is different. Indeed, the dual of $BV(\mathbb{R}^N)$ is not well known yet and because of that, it is not known whether or not the space $BV(\mathbb{R}^N)$ is reflexive. This becomes a very difficult task to find critical points of functionals defined in this space and, as a consequence, we can see few or even any work dealing with elliptic problems in $\mathbb{R}^N$ which are normally modeled in this space.

In this work we are interested in the following quasilinear problem
\begin{equation}
\left\{
\begin{array}{rl}
\displaystyle - \Delta_1 u + V(x)\frac{u}{|u|} & = K(x)f(u) \quad \mbox{in
$\mathbb{R}^N$,}\\
u(x) \to 0 & \mbox{as $|x| \to +\infty.$}
\end{array} \right.
\label{Pintro}
\end{equation}
where the differential operator $1-$Laplacian is formally defined as $\displaystyle \Delta_1 u = \mbox{div}\left(\frac{\nabla u}{|\nabla u|}\right)$. The nonlinearity $f$ is assumed to satisfy the following set of assumptions:
\begin{itemize}
\item [$(f_1)$] $f \in C^0(\mathbb{R})$;
\item [$(f_2)$] $f(s) = o(|s|)$ as $s \to 0$;
\item [$(f_3)$] $f$ has a {\it quasicritical growth}, i.e., 
$$\displaystyle \limsup_{s \to +\infty}\frac{f(s)}{s^{{1^*}-1}} = 0,$$
where $\displaystyle 1^*:= \frac{N}{N-1}$;
\item [$(f_4)$] There exists $\theta > 1 $ such that $$0 < \theta F(s) \leq f(s)s, \quad \mbox{for $s \neq 0$},$$
where $F(s) = \int_0^s f(t)dt$;
\item [$(f_5)$] $f$ is increasing in $\mathbb{R}$.
\end{itemize}

The potentials $V$ and $K$ are assumed to satisfy some assumptions. We say that $(V, K) \in \mathcal{K}$ if the following conditions hold:
\begin{description}
\item[$(VK_0)$] $V(x), K(x) > 0$ for all $x \in \mathbb R^N$, $K \in C^0(\mathbb{R}^N) \cap L^{\infty}(\mathbb R^N)$;
\item[$(VK_1)$]   If $(A_n) \subset \mathbb R^N$ is a sequence of Borel sets such that its Lebesgue measure $|A_n| \leq R$, for all $n \in \mathbb N$ and some $R > 0$, then
$$\lim_{r \rightarrow + \infty} \int_{A_n \cap B_r^c(0)} K(x) =0, \quad\hbox{uniformly in $n \in \mathbb N$}.  $$
\end{description}
Furthermore, one of the below conditions occurs
\begin{description}
\item[$(VK_2)$]   $\displaystyle  \frac{K}{V} \in L^{\infty}(\mathbb R^N)$
\end{description}
or
\begin{description}
\item[$(VK_3)$] there exists $q \in (1, 1^*)$  such that
$$  \frac{K(x)}{V(x)^\frac{1^*-q}{1^*-1}} \rightarrow 0 \quad\hbox{as $|x| \rightarrow + \infty$}. $$
\end{description}

This kind of assumptions in the potentials has been considered for the first time in \cite{AlvesSouto}, where the authors succeed in showing the existence of ground-state solutions to semilinear problems by considering vanishing potentials.

Since $\displaystyle \inf_{\mathbb{R}^N}V$ can be equal to $0$, it is not possible to work with $BV(\mathbb{R}^N)$ itself endowed with any norm including $V$. To overcome this difficulty we have to work with a space that plays the same role that $D^{1,2}(\mathbb{R}^N)$, when working with semilinear problems. This space is defined by
$$
E = \left\{u \in L^{1^*}(\mathbb{R}^N); \, \int_{\mathbb{R}^N}|Du| < \infty \right\},
$$
and it seems that this is the first work in which this space is defined and studied. Because of that, for the sake of completeness in Section 2 we are prove the properties that this space in fact have, like its completeness, etc.

By using a variational approach, in this work we have to deal with an Euler-Lagrange functional which is not smooth, although locally Lipschitz. Then the way in which the functional and its Euler-Lagrange equation is linked is somehow tricky. In fact the sense of solution we consider here has to take into account the concept of generalized gradient developed by Clarke (see \cite{Clarke,Chang}). More precisely, the Euler-Lagrange functional associated with (\ref{Pintro}) is given by
$$
\Phi(u) = \int_{\mathbb{R}^N}|Du| + \int_{\mathbb{R}^N}V(x)|u|dx - \int_{\mathbb{R}^N}K(x)F(u),
$$
where $Du$ is the distributional derivative of $u$, which in turn is a Radon measure. As can be seen in Section 2 we say that $u$ is a bounded variation solution of (\ref{Pintro}) if
$$
\mathcal{J}(v) - \mathcal{J}(u) \geq \int_{\mathbb{R}^N}K(x)f(u)(v-u) dx,
$$
for all $v$ in the domain of $\Phi$, where
$$
\mathcal{J}(u) = \int_{\mathbb{R}^N}|Du| + \int_{\mathbb{R}^N}V(x)|u|dx.
$$

The main result of our paper, which seems to be together with \cite{PimentaFigueiredo1} the very first works dealing with the $1-$Laplacian operator in $\mathbb{R}^N$, is the following result stating the existence of a ground-state bounded variation solution of (\ref{Pintro}).

\begin{theorem}
Suppose that $f$ satisfies  the conditions $(f_1) - (f_5)$ and $(V,K) \in \mathcal{K}$. Then there exists a ground-state bounded variation solution $u$ of (\ref{Pintro}). Moreover, $u$ satisfies the following Euler-Lagrange equation (which is the precise version of (\ref{Pintro}))
\begin{equation}
\left\{
\begin{array}{l}
\exists z \in L^\infty(\mathbb{R}^N,\mathbb{R}^N), \, \, \|z\|_\infty \leq 1,\, \,  \mbox{div}z \in L^N(\mathbb{R}^N), \, \, -\int_{\mathbb{R}^N}u \mbox{div}z dx = \int_{\mathbb{R}^N}|Du|,\\
\exists z_2^* \in L^N(\mathbb{R}^N),\, \, z_2^*|u| = V(x) u \quad \mbox{a.e. in $\mathbb{R}^N$},\\
-\mbox{div} z + z_2^* = f(u), \quad \mbox{a.e. in $\mathbb{R}^N$}.
\end{array}
\right.
\label{eulerlagrangeintro}
\end{equation}
\label{theoremapplication}
\end{theorem}

A main point in proving Theorem \ref{theoremapplication} is a sort of compactness result. In fact we take benefit from the conditions in $V$ and $K$ and prove the compactness of the embedding of a subspace of $\mathbb{R}^N$ in some weighted Sobolev spaces (see Proposition \ref{propositioncompactness}).

It is worth to highlight again some of the peculiarities of working with functionals in the space of bounded variation functions in $\mathbb{R}^N$ or even in smooth bounded domains. In general, when working with functionals in usual Sobolev spaces, compactness of embeddings is symnonimous of the validity of the Palais Smale condition. However, when dealing with a nonreflexive space like $E$ or $BV(\Omega)$, just the compactness of the embeddings is not enough to ensure the strong convergence of the Palais Smale sequences. This happens since in a nonreflexive setting, the convergence of the norm of a Palais Smale sequence, does not imply the strong convergence of the sequence itself. This clearly requires the application of versions of minimax theorems without the Palais Smale condition among their assumptions, even in problems in bounded domains. This is in strike contrast with the regular case and also is why in general, when dealing with problems involving $1-$Laplacian or mean-curvature operators in $BV(\Omega)$, the standard technique that has been applied is to extend the Euler-Lagrange functional to some Lebesgue space, gaining reflexivity of the domain (and Palais-Smale condition as a consequence), but loosing continuity of the functional. Because of these difficulties we present in an Appendix a version of the Mountain Pass Theorem to locally Lipschitz functionals and also of the Deformation Lemma, both in the absense of the Palais-Smale condition, since we could not find them in the literature.

A brief explanation of some points about the highly singular nature of the operator $\Delta_1$ is in order. The $1-$Laplacian operator can be seen as the formal limit of the usual $p-$Laplacian operator, as $p \to 1^+$ (see \cite{Demengel} for instance) and have several applications like in image restoration (see \cite{Chen}), game theory (see \cite{Evans}), etc. If we get a closer look at the definition of the $1-$Laplacian operator, we clearly see that it is not well defined wherever $\nabla u$ vanishes. This becomes so much relevant, since in general solutions of problems involving the $\Delta_1$ operator vanishes in positive measure sets, like its eigenfunctions in bounded sets, as shown in \cite{Milbers}, \cite{Kawohl} and \cite{Chang2}.

As far as the approach one can follow in studying $1-$Laplacian problems in bounded domains, roughly speaking there are two ways to follow. One can study $\Delta_1$ through $p-$Laplacian problems and then taking the limit as $p \to 1^+$, like in \cite{Demengel,Demengel1}, or one can directly deal with $\Delta_1$ itself, by using variational methods for instance. However, to precisely understand (\ref{Pintro}), we have to replace the expression $\displaystyle \frac{\nabla u}{|\nabla u|}$ by a well defined vector field which extend the former wherever $\nabla u$ vanishes. Similarly, one has to substitute $\displaystyle\frac{u}{|u|}$ by some function to give meaning to this expression wherever $u$ vanishes. In Section 2 we pay more attention to this issue, showing that (\ref{eulerlagrangeintro}) is the precise Euler-Lagrange equation satisfied by the bounded variation critical points of the energy functional $\Phi$.

A version of Br\'ezis-Nirenberg problem to $1-$Laplacian operator has been studied in \cite{DegiovanniMagrone}, where the authors use a nonstandard linking structure in order to get solutions of the problem. In \cite{LeonWebler} the authors study a parabolic problem involving the $1-$Laplacian operator and succeed in proving global existence and uniqueness for source and initial data in some adequate space. In \cite{FigueiredoPimenta} the authors seem to be pioneers in using Nehari types arguments in order to get bounded variation solutions for problems involving the mean-curvature or the $1-$Laplacian operators. In \cite{PimentaFigueiredo1}, the authors proved versions of several classical results in a $BV$ setting, like Strauss Radial Lemma, the compactness of the embedding $BV_{rad}(\mathbb{R}^N)$ in $L^q(\mathbb{R}^N)$ for $q \in (1,1^*)$ and also a version of Lions' Lemma. An application of these results to find a ground-state solution of a $1-$Laplacian problem in $\mathbb{R}^N$ is also presented in that paper.

The paper is organized as follows. In Section 2 we present the variational framework, study the precise Euler-Lagrange equation associated to $\Phi$ and prove a crucial compactness result. In Section 3 we prove the result about the existence of a bounded variation solution of (\ref{Pintro}). In the Appendix we prove the version of the Mountain Pass Theorem and of the Deformation Lemma we are using here.

\section{Variational framework and the Euler-Lagrange equation}

\subsection{The space and the energy functional}

\hspace{.5cm}  The space we deal in this work is the space of functions in $L^{1^*}(\mathbb{R}^N)$, such that its distributional derivative $Du$ is a vectorial Radon measure, i.e.,
$$
E = \left\{ u \in L^{1^*}(\mathbb{R}^N); \, Du \in \mathcal{M}(\mathbb{R}^N,\mathbb{R}^N) \right\},
$$
where $\mathcal{M}(\mathbb{R}^N,\mathbb{R}^N)$ is the space of the vectorial Radon measures. First, let us consider the following norm in $E$, 
$$
\|u\|_0:= \int_{\mathbb{R}^N}|Du| + |u|_{1^*}.
$$

It can be proved that $u \in E$ is equivalent to $u \in L^{1^*}(\mathbb{R}^N)$ and also
$$\int_{\mathbb{R}^N} |Du| := \sup\left\{\int_{\mathbb{R}^N} u \mbox{div}\phi dx; \, \, \phi \in C^1_c(\mathbb{R}^N,\mathbb{R}^N), \, \mbox{s.t.} \, \, |\phi|_\infty \leq 1\right\} < +\infty.$$

This space seems to be related with the space of functions of bounded variation of $\mathbb{R}^N$, $BV(\mathbb{R}^N)$, in the same way that $D^{1,2}(\mathbb{R}^N)$ is related with $H^1(\mathbb{R}^N)$. Since we could not even find the definition of this space in the literature, below we prove some of its properties. In fact they are inspired in some results in \cite{Buttazzo}[Chapter 10] and we just include them here for the sake of completeness.

\begin{lemma}
Let $(u_n) \in E$ be a bounded sequence such that $u_n \to u$ in $L^{1^*}(\mathbb{R}^N)$. Then it follows that $u \in E$ and
$$\int_{\mathbb{R}^N} |Du| \leq \liminf_{n \to \infty} \int_{\mathbb{R}^N} |Du_n|.$$
\label{semicontinuity}
\end{lemma}
\begin{proof}
Let $\phi \in C^1_c(\mathbb{R}^N,\mathbb{R}^N)$ such that $|\phi|_\infty \leq 1$ and note that
$$\int_{\mathbb{R}^N}u \mbox{div}\phi dx = \int_{supp(\phi)}u \mbox{div}\phi dx = \lim_{n \to +\infty}\int_{supp(\phi)}u_n \mbox{div}\phi dx \leq \liminf_{n \to \infty}\int_{\mathbb{R}^N}|Du_n|,$$
where the second equality follows since $u_n \to u$ in $L^{1^*}(supp(\phi))$ implies that $u_n \to u$ in $L^1(supp(\phi))$. Then the result follows by taking the supremum over all such $\phi$.
\end{proof}

Now let us prove that $(E, \|\cdot\|_0)$ is a Banach space.
\begin{proposition}
$(E, \|\cdot\|_0)$ is a Banach space
\end{proposition}
\begin{proof}
Let $(u_n) \subset E$ be a Cauchy sequence. Then for all $\epsilon > 0$, there exists $n_0 \in \mathbb{N}$ such that
\begin{equation}
\int_{\mathbb{R}^N}|Du_n - Du_m| + |u_n - u_m|_{1^*} < \epsilon, \quad n,m \geq n_0. \label{nm}
\end{equation}
Since $(u_n)$ is also a Cauchy sequence in $L^{1^*}(\mathbb{R}^N)$, there exists $u$ such that $u_n \to u$ in $L^{1^*}(\mathbb{R}^N)$, as $n \to +\infty$. Then, for all $m \in \mathbb{N}$ such that $m \geq n_0$, since $u_n - u_m \to u - u_m$ as $n \to +\infty$, Lemma \ref{semicontinuity} implies that
$$
\int_{\mathbb{R}^N}|D(u - u_m)| \leq \liminf_{n \to \infty}\int_{\mathbb{R}^N}|D(u_n - u_m)| < \epsilon.
$$
But this implies that $u \in E$ and also that $u_m \to u$ in $E$, as $m \to +\infty$.
\end{proof}

The following result states the density of $C^{\infty}_0(\mathbb{R}^N)\cap E$ in $E$ with respect to a notion of convergence which resembles the intermediate convergence in $BV$ (see \cite{Buttazzo}[Chapter 10]), i.e., $\displaystyle E = \overline{C^{\infty}_0(\mathbb{R}^N)\cap E}^{\tau}$ where $\tau$ is the topology induced by the notion of convergence given by (\ref{inter1}) and (\ref{inter2}) (see below). This result is going to be crucial in order to prove a version of the Sobolev inequality in the space $E$.

\begin{proposition}
For all $u \in E$, there exists $(u_n) \subset C^{\infty}_0(\mathbb{R}^N)\cap E$ such that
\begin{eqnarray}
u_n \to u \quad \mbox{in $L^{1^*}(\mathbb{R}^N)$,} \label{inter1}\\
\int_{\mathbb{R}^N}|\nabla u_n|dx \to \int_{\mathbb{R}^N}|D u|. \label{inter2}
\end{eqnarray}
\label{density}
\end{proposition}
\begin{proof}
Given $\epsilon > 0$, let $R > 0$ be such that
\begin{equation}
\int_{B_R(0)^c}|Du| < \epsilon. \label{inter3}
\end{equation}
Let $\{\varphi_1,\varphi_2\}$ be a partition of the unity subordinated to the open covering $\left\{B_{R+\delta}(0), \overline{B_{R-\delta}(0)}^c\right\}$ of $\mathbb{R}^N$, where $\delta > 0$ is fixed. Then note that $\varphi_1 \in C^\infty_c(B_{R+\delta}(0))$, $\varphi_2 \in C^\infty_c(\overline{B_{R-\delta}(0)}^c)$, $0 \leq \varphi_i \leq 1$, $i = 1,2$ and $\varphi_1 + \varphi_2 = 1$ in $\mathbb{R}^N$. Let us suppose also without lack of generality that $\varphi_1 \equiv 1$ in $B_R(0)$ and $\varphi_2 \equiv 1$ in $\overline{B_R(0)}^c$. Let $\rho_1, \rho_2 \in C^\infty_c(\mathbb{R}^N)$ two mollifiers such that
\begin{equation}
supp(\rho_1 \star (\varphi_1 u)) \subset B_{R+\delta}(0), \label{mollifier1}
\end{equation}
\begin{equation}
supp(\rho_2 \star (\varphi_2 u)) \subset \overline{B_{R-\delta}(0)}^c, \label{mollifier2}
\end{equation}
\begin{equation}
\int_{\mathbb{R}^N}|\rho_i \star(u\varphi_i) - u\varphi_i|^{1^*}dx < \epsilon, \quad  i = 1,2 \label{mollifier3}
\end{equation}
\begin{equation}
\int_{\mathbb{R}^N}|\rho_i \star(u\nabla\varphi_i) - u\nabla\varphi_i|dx < \epsilon,  \quad  i = 1,2 \label{mollifier4}
\end{equation}
\begin{equation}
\int_{\mathbb{R}^N}|\rho_1\star(\varphi_1Du)|dx - \int_{\mathbb{R}^N}|\varphi_1Du| < \epsilon.  \label{mollifier5}
\end{equation}

Let us define
\begin{equation}
u_\epsilon = \rho_1 \star (u\varphi_1) + \rho_2 \star (u\varphi_2)
\label{u_epsilon}
\end{equation}
and note that clearly $u_\epsilon \in C_0^\infty(\mathbb{R}^N)$.
Also, since $\varphi_1 + \varphi_2 = 1$ in $\mathbb{R}^N$, by (\ref{mollifier3})
\begin{eqnarray*}
\int_{\mathbb{R}^N}|u - u_\epsilon|^{1^*}dx & \leq & \int_{\mathbb{R}^N}|u \varphi_1 - \rho_1\star(u\varphi_1)|^{1^*}dx + \int_{\mathbb{R}^N}|u \varphi_2 - \rho_1\star(u\varphi_2)|^{1^*}dx\\
& < & 2\epsilon
\end{eqnarray*}
and this implies in (\ref{inter1}).

Since $\nabla \varphi_1 + \nabla \varphi_2 \equiv 0$ in $\mathbb{R}^N$, it follows that
\begin{eqnarray*}
\nabla u_\epsilon & = & \sum_{i = 1}^2\rho_i\star (\varphi_iDu)  + \sum_{i = 1}^2\rho_i\star (u\nabla \varphi_i)\\
& = & \sum_{i = 1}^2\rho_i\star (\varphi_iDu)  + \sum_{i = 1}^2(\rho_i\star (u\nabla \varphi_i) - u\nabla \varphi_i).
\end{eqnarray*}
Then (\ref{mollifier4}) and (\ref{inter3}) imply that
\begin{eqnarray*}
\int_{\mathbb{R}^N}|\nabla u_\epsilon|dx - \int_{\mathbb{R}^N}|\rho_1\star(\varphi_1Du)| dx& \leq & \int_{\mathbb{R}^N}|\rho_2\star(\varphi_2Du)|dx + \sum_{i=1}^2 \int_{\mathbb{R}^N}|\rho_i \star(u\nabla\varphi_i) - u\nabla\varphi_i|dx\\
& \leq & \int_{\mathbb{R}^N}|\varphi_2Du| + 2\epsilon\\
& \leq & \int_{B_R(0)^c}|Du| + 2\epsilon\\
& < & 3\epsilon.
\end{eqnarray*}
Hence (\ref{mollifier5}) and last inequality imply that
\begin{eqnarray*}
\left| \int_{\mathbb{R}^N}|\nabla u_\epsilon|dx - \int_{\mathbb{R}^N}|Du| \right| & \leq &  \left| \int_{\mathbb{R}^N}|\nabla u_\epsilon|dx - \int_{\mathbb{R}^N}|\rho_1\star(\varphi_1Du)|dx\right|\\
& &  + \left| \int_{\mathbb{R}^N}|\rho_1\star(\varphi_1Du)|dx -  \int_{\mathbb{R}^N}|Du| \right|\\
& < & 3\epsilon + \left| \int_{\mathbb{R}^N}|\rho_1\star(\varphi_1Du)|dx - \int_{\mathbb{R}^N}|\varphi_1Du|\right|\\
& &  +  \left| \int_{\mathbb{R}^N}|\varphi_1Du| - \int_{\mathbb{R}^N}|Du|\right|\\
& < & 4\epsilon + \int_{\mathbb{R}^N}|\varphi_1 - 1||Du|\\
& \leq & 4\epsilon + \int_{B_R(0)^c}|Du|\\
& < & 5\epsilon.
\end{eqnarray*}
This implies that $u_\epsilon \in C^\infty_0(\mathbb{R}^N)\cap E$ and also in (\ref{inter2}) and finish the proof.
\end{proof}

Now we prove that the Sobolev inequality holds in $E$.

\begin{proposition}
There exists $C > 0$ such that, for all $u \in E$
\begin{equation}
|u|_{1^*} \leq C\int_{\mathbb{R}^N}|Du|.
\label{sobolevinequality}
\end{equation}
\end{proposition}
\begin{proof}
For $u \in E$, by Proposition \ref{density}, there exist $(u_n) \subset C^\infty_0(\mathbb{R}^N)$ such that (\ref{inter1}) and (\ref{inter2}) hold. Then Fatou Lemma and Gagliardo-Nirenberg inequality applied for $u_n$ imply that
\begin{eqnarray*}
|u|_{1^*} & \leq & \liminf_{n \to \infty}|u_n|_{1^*}\\
& \leq & \liminf_{n \to \infty}C\int_{\mathbb{R}^N}|\nabla u_n|dx\\
& = & C\int_{\mathbb{R}^N}|Du|
\end{eqnarray*}
and proves the result.
\end{proof}

By the last result, the following norm in $E$ defined by
$$\|u\|_E = \int_{\mathbb{R}^N}|Du|$$
is equivalent to $\|\cdot\|_0$ and then $(E,\|\cdot\|_E)$ is also a Banach space.

For a vectorial Radon measure $\mu \in \mathcal{M}(\mathbb{R}^N,\mathbb{R}^N)$, we denote by $\mu = \mu^a + \mu^s$ the usual decomposition stated in the Radon Nikodyn Theorem, where $\mu^a$ and $\mu^s$ are, respectively, the absolutely continuous and the singular parts with respect to the $N-$dimensional Lebesgue measure $\mathcal{L}^N$. We denote by $|\mu|$, the absolute value of $\mu$, the scalar Radon measure defined like in \cite{Buttazzo} (pg. 125). By $\displaystyle \frac{\mu}{|\mu|}(x)$ we denote the usual Lebesgue derivative of $\mu$ with respect to $|\mu|$, given by
$$\frac{\mu}{|\mu|}(x) = \lim_{r \to 0}\frac{\mu(B_r(x))}{|\mu|(B_r(x))}.$$ 

Let us define 
$$
E_V = \left\{u \in E; \, \int_{\mathbb{R}^N}V(x)|u|dx < \infty\right\}
$$
endowed with the following norm which turns it a Banach space
$$
\|u\| : = \int_{\mathbb{R}^N}|Du| + \int_{\mathbb{R}^N}V(x)|u|dx.
$$

It can be proved that $\mathcal{J}: E_V \to \mathbb{R}$, given by
\begin{equation}
\mathcal{J}(u) = \int_{\mathbb{R}^N} |Du| + \int_{\mathbb{R}^N} V(x)|u|dx,
\label{J}
\end{equation}
is a convex functional and Lipschitz continuous in its domain.

It can be proved that $E_V$ is a {\it lattice}, i.e., if $u,v \in E_V$, then $\max\{u,v\}, \min\{u,v\} \in E_V$ and also
\begin{equation}
\mathcal{J}(\max\{u,v\}) + \mathcal{J}(\min\{u,v\}) \leq \mathcal{J}(u) + \mathcal{J}(v), \quad \forall u,v \in E_V.
\label{lattice}
\end{equation}

Although non-smooth, the functional $\mathcal{J}$ admits some directional derivatives. More specifically, as is shown in \cite{Anzellotti}, given $u \in E_V$, for all $v \in E_V$ such that $(Dv)^s$ is absolutely continuous with respect to $(Du)^s$, it follows that
$$
\mathcal{J}'(u)v = \int_{\mathbb{R}^N} \frac{(Du)^a(Dv)^a}{|(Du)^a|}dx + \int_{\mathbb{R}^N} \frac{Du}{|Du|}(x)\frac{Dv}{|Dv|}(x)|(Dv)|^s + \int_{\mathbb{R}^N}V(x)\mbox{sgn}(u) v dx.
$$
In particular, note that
\begin{equation}
\mathcal{J}'(u)u = \mathcal{J}(u).
\label{derivadaJ}
\end{equation}

Now let us present the energy functional associated to (\ref{Pintro}).

Let $\Phi: E_V \to \mathbb{R}$ be given by
\begin{equation}
\Phi(u) = \mathcal{J}(u) - \mathcal{F}(u).
\label{Phi}
\end{equation}
where $\mathcal{F}: E_V \to \mathbb{R}$ is defined by
\begin{equation}
\mathcal{F}(u) = \int_{\mathbb{R}^N}K(x)F(u)dx.
\label{F}
\end{equation}
It is a simple matter to prove that $\mathcal{F} \in C^1(E_V, \mathbb{R})$.

Now let us make precise the sense of solution we are considering here. Since $\Phi$ can be written as the difference between a Lipschitz and a smooth functional, we say that $u_0 \in E_V$ is a solution of (\ref{Pintro}) if $0 \in \partial \Phi(u_0)$, where $\partial \Phi(u_0)$ denotes the generalized gradient of $\Phi$ in $u_0$, as defined in \cite{Chang}. It follows that this is equivalent to $\mathcal{F}'(u_0) \in \partial \mathcal{J}(u_0)$ and, since $\mathcal{J}$ is convex, this can be written as
\begin{equation}
\mathcal{J}(v) - \mathcal{J}(u_0) \geq \mathcal{F}'(u_0)(v-u_0), \quad \forall v \in E_V.
\label{BVsolution}
\end{equation}
Hence all $u_0 \in E_V$ such that (\ref{BVsolution}) holds is going to be called a bounded variation solution of (\ref{Pintro}). 

\subsection{The Euler-Lagrange equation}
\label{eulerlagrangesection}

\hspace{.5cm} As we said in the introduction,  problem (\ref{Pintro}) is just the formal version of the Euler-Lagrange equation associated to the functional $\Phi$, since contains expressions that doesn't make sense when $\nabla u = 0$ or $u = 0$. In order to present a precise form of an Euler-Lagrange equation satisfied by all bounded variation critical points of $\Phi$, let us follow the arguments in \cite{Kawohl}.

First of all let us consider the extension of the functionals $\mathcal{J}, \mathcal{F}$ and $\Phi$ to $L^{1^*}(\mathbb{R}^N)$ given by $\overline{\mathcal{J}}, \overline{\mathcal{F}}, \overline{\Phi}: L^{1^*}(\mathbb{R}^N) \to \mathbb{R}$, where
$$
\overline{\mathcal{J}}(u) = 
\left\{
\begin{array}{ll}
\mathcal{J}(u), & \mbox{if $u \in E_V$},\\
+\infty, & \mbox{if $u \in L^{1^*}(\mathbb{R}^N)\backslash E_V$},
\end{array}
\right.
$$
$$
\overline{\mathcal{F}}(u) = \int_{\mathbb{R}^N}K(x)F(u)dx
$$
and $\overline{\Phi} = \overline{\mathcal{J}} - \overline{\mathcal{F}}$. Note that $\overline{\mathcal{F}}$ belongs to $C^1(L^{1^*}(\mathbb{R}^N), \mathbb{R})$ and also $\overline{\mathcal{J}}$ is a convex lower semicontinuous functional in $L^{1^*}(\mathbb{R}^N)$. Then the subdifferential (in the sense of \cite{Szulkin}) of $\overline{\mathcal{J}}$, denoted by $\partial \overline{\mathcal{J}}$, is well defined. The following result is crucial to the construction of such Euler-Lagrange equation.

\begin{lemma}
If  $u \in E_V$ is such that $0 \in \partial \Phi(u)$, then $0 \in \partial \overline{\Phi}(u)$.
\end{lemma}
\begin{proof}
In fact, let $u \in E_V$ such that $0 \in \partial \Phi(u)$ and note that $u$ satisfies (\ref{BVsolution}).
We would like to prove that
$$
\overline{\mathcal{J}}(v) - \overline{\mathcal{J}}(u) \geq \overline{\mathcal{F}}\, '(u)(v-u), \quad \forall v \in L^{1^*}(\mathbb{R}^N).
$$
Let $v \in L^{1^*}(\mathbb{R}^N)$ and note that:
\begin{itemize}
\item if $v \in E_V \cap L^{1^*}(\mathbb{R}^N)$, then
\begin{eqnarray*}
\overline{\mathcal{J}}(v) - \overline{\mathcal{J}}(u) & = & \mathcal{J}(v) - \mathcal{J}(u)\\
& \geq & \mathcal{F}'(u)(v-u)\\
& = & \int_{\mathbb{R}^N}f(u)(v-u)dx\\
& = & \overline{\mathcal{F}}\, '(u)(v-u);
\end{eqnarray*}

\item if $u \in L^{1^*}(\mathbb{R}^N)\backslash E_V$, since $\overline{\mathcal{J}}(v) = +\infty$ and $\overline{\mathcal{J}}(u) < +\infty$, it follows that
\begin{eqnarray*}
\overline{\mathcal{J}}(v) - \overline{\mathcal{J}}(u) & = & +\infty\\
& \geq & \overline{\mathcal{F}}\, '(u)(v-u).
\end{eqnarray*}
\end{itemize}
Therefore the result follows.
\end{proof}

Now suppose that $u \in E_V$ is a bounded variation solution of (\ref{Pintro}), i.e., that $u$ satisfies (\ref{BVsolution}). Since $0 \in \partial \Phi(u)$, by the last result it follows that $0 \in \partial \overline{\Phi}(u)$. Since $\overline{\mathcal{J}}$ is convex and $\overline{\mathcal{F}}$ is smooth, it follows that $\overline{\mathcal{F}}\, '(u) \in \partial \overline{\mathcal{J}}(u)$. Now, defining $\overline{\mathcal{J}_1}(u) := \int_{\mathbb{R}^N} |Du|$ and $\overline{\mathcal{J}_2}(u) := \int_{\mathbb{R}^N} V(x)|u|dx$, it follows that
$$
\overline{\mathcal{F}}\, '(u) \in \partial \overline{\mathcal{J}}(u) \subset \partial \overline{\mathcal{J}_1}(u) + \partial \overline{\mathcal{J}_2}(u).
$$
This means that there exist $z_1^*, z_2^* \in L^N(\mathbb{R}^N)$ such that $z_1^* \in \partial \overline{\mathcal{J}_1}(u)$, $z_2^* \in \partial \overline{\mathcal{J}}_2(u)$ and
$$
\overline{\mathcal{F}}\, '(u) = z_1^* +z_2^* \quad \mbox{ in $L^N(\mathbb{R}^N)$.}
$$
Repeating the arguments of \cite{Kawohl}[Proposition 4.23, pg. 529], it follows that there exist $z \in L^\infty(\mathbb{R}^N, \mathbb{R}^N)$ such that $\|z\|_\infty \leq 1$,
\begin{equation}
-\mbox{div}{z} = z_1^* \quad \mbox{ in $L^N(\mathbb{R}^N)$}
\label{eulerlagrange1}
\end{equation}
and 
\begin{equation}
 -\int_{\mathbb{R}^N}u \mbox{div}z dx = \int_{\mathbb{R}^N}|Du|,
 \label{eulerlagrange2}
 \end{equation}
 where the divergence in (\ref{eulerlagrange1}) has to be understood in the distributional sense. Moreover, the same result implies that $z_2^*$ is such that 
\begin{equation}
z_2^* |u| = V(x)u, \quad \mbox{a.e. in $\mathbb{R}^N$.}
\label{eulerlagrange3}
\end{equation}
Therefore, from (\ref{eulerlagrange1}), (\ref{eulerlagrange2}) and (\ref{eulerlagrange3}) it follows that $u$ satisfies
\begin{equation}
\left\{
\begin{array}{l}
\exists z \in L^\infty(\mathbb{R}^N,\mathbb{R}^N), \, \, \|z\|_\infty \leq 1,\, \,  \mbox{div}z \in L^N(\mathbb{R}^N), \, \, -\int_{\mathbb{R}^N}u \mbox{div}z dx = \int_{\mathbb{R}^N}|Du|,\\
\exists z_2^* \in L^N(\mathbb{R}^N),\, \, z_2^*|u| = V(x) u \quad \mbox{a.e. in $\mathbb{R}^N$},\\
-\mbox{div} z + z_2^* = f(u), \quad \mbox{a.e. in $\mathbb{R}^N$}.
\end{array}
\right.
\end{equation}

\subsection{A compactness result}

\hspace{.5cm} Let us define the weighted Lebesgue space
$$
L^q_K(\mathbb{R}^N) = \left\{u: \mathbb{R}^N \to \mathbb{R}; \, \mbox{$u$ is measurable and $\displaystyle \int_{\mathbb{R}^N}K(x)|u|^qdx < +\infty$}\right\},
$$
which is Banach spaces when endowed with the norm
$$
|u|_{q,K}: = \left(\int_{\mathbb{R}^N}K(x)|u|^qdx\right)^\frac{1}{q}.
$$

Let us prove a compactness result which is going to be crucial in our argument.

\begin{proposition}
Assume that $(V,K) \in \mathcal{K}$. Then $E_V$ is compactly embedded in $L^r_K(\mathbb{R}^N)$ for all $r \in (1,1^*)$ if $(K_2)$ holds. If $(K_3)$ holds, then $E_V$ is compactly embedded in $L^q_K(\mathbb{R}^N)$, where $q$ is given in $(VK_3)$.
\label{propositioncompactness}
\end{proposition}

\begin{proof}
Let $(u_n) \subset E_V$ a bounded sequence and let $C > 0$ be such that 
$$
\|u_n\| \leq C, \quad \forall n \in \mathbb{N}.
$$
Suppose first that $(VK_2)$ holds. Fixed $q \in
(1, 1^*)$ and $\epsilon > 0$, there exist $0 < t_0 <
t_1$ and a positive constant $C > 0$ such that
\begin{equation*}
K(x)|t|^q \leq \epsilon\,  C 	\,  (V(x) |t| + |t|^{1^*}) + C \, 
K(x) \, \chi_{[t_0, t_1]}(|t|) |t|^{1^*}, \quad\hbox{ for all $ t \in
\mathbb R$.}
\end{equation*}
Then, denoting $Q(u)= \displaystyle\int_{\mathbb{R}^N} V(x) |u|
dx + \displaystyle\int_{\mathbb{R}^N} |u|^{1^*} dx $ and
$A_n= \left\{ x \in \mathbb{R}^N: t_0 \leq |u_n(x)| \leq t_1   \right\}$,
we have that
\begin{equation}\label{2.5}
\int_{B_r(0)^c} K(x) |u_n|^q dx \leq \epsilon \,  C Q(u_n) + C
\int_{A_n \cap B_r(0)^c} K(x) dx, \quad\hbox{for all $n \in \mathbb{N}$.}
\end{equation}
By the boundedness of $(u_n)$ and the continuity of the embedding $E_V \hookrightarrow L^{1^*}(\mathbb{R}^N)$, there exists $C' > 0$ such that
$Q(u_n) \leq C''$ for all $n \in \mathbb{N}$. On the other hand it follows
that
$$ t_0^{1^*} |A_n| \leq \int_{A_n} |u_n|^{1^*} dx \leq C', \quad\hbox{for any $n \in \mathbb N$},$$
and then $\sup_{n \in \mathbb N} |A_n| < + \infty$. Consequently, from $(VK_1)$ there exists a positive radius $r > 0$ large enough such that
\begin{equation}\label{2.6}
\int_{A_n \cap B_r^c(0)} K(x) dx < \frac{\epsilon}{t_1^{1^*}} \quad\hbox{for all $n \in \mathbb N$}.
\end{equation}
By \eqref{2.5} and \eqref{2.6} it follows that
\begin{eqnarray}\label{2.7}
\int_{B_r(0)^c} K(x) |u_n|^q dx &\leq&CC''\epsilon 
+ C \int_{A_n \cap B_r(0)^c} K(x) dx \nonumber \\
&\leq &
(CC'' + C/{t_1^{1^*}}) \epsilon \quad\hbox{for all $n \in \mathbb N$}.
\end{eqnarray}
Denoting $E_r :=\left\{u \in BV(B_r(0)); \, \int_{B_r(0)}V(x)|u|dx < \infty\right\}$, since $\displaystyle \left(\left. u_n\right|_{B_R(0)}\right)_{n \in \mathbb{N}}$ is a bounded sequence in $E_r$, the compactness of the embedding $E_r \hookrightarrow L^q(B_r(0))$, for $q \in [1,1^*)$ and the boundedness of $K$ imply that there exists $u \in E_r$ such that
\begin{equation}\label{2.8}
\lim_{n \rightarrow + \infty } \int_{B_r(0)} K(x) |u_n|^q dx
= \int_{B_r(0)} K(x) |u|^q dx.
\end{equation}
Now, denoting by $u$ its own extension by zero in $B_r(0)^c$, note that $u \in E_V$ and, for $\epsilon > 0$, if $n$ is sufficiently large, by \eqref{2.7} and \eqref{2.8} it follows that
\begin{equation}
\int_{\mathbb R^N} K(x) |u_n - u|^q dx = \int_{B_r(0)} K(x) |u_n - u|^qdx + \int_{B_r(0)^c} K(x) |u_n|^q dx < \epsilon,
\label{eq}
\end{equation}
which concludes the proof in this case.

\medskip

Now suppose that $(VK_3)$ holds. Let us define for each $x \in \mathbb R^3$ fixed, the function
$$g(t)= V(x) t^{1-q} + t^{1^*-q}, \quad\hbox{for every $t > 0$}. $$
Since its minimum value is $C_q V(x)^\frac{1^*-q}{1^*-1}$ where 
$$C_q = \left(\frac{q-1}{1^*-q}\right)^\frac{1-q}{1^*-1} + \left(\frac{q-1}{1^*- q}\right)^\frac{1^*-q}{1^*-1},$$
we have that
$$C_q V(x)^\frac{1^*-q}{1^*-1} \leq V(x) t^{1 - q} + t^{1^*- q}, \quad\hbox{for every $x \in \mathbb R^N$ and $t > 0$}. $$
Combining this inequality  with $(VK_3)$, for any $\epsilon > 0$ 
there exists a positive radius $r > 0$ sufficiently large
such that
\begin{equation*}
 K(x) |t|^q \leq \epsilon \, C'_q (V(x) |t| + |t|^{1^*}), \quad\hbox{for every $t \in \mathbb R$ and $|x| > r$,}
 \end{equation*}
where $C'_q = C_q^{-1}$,  from which it follows that
$$ \displaystyle\int_{B_r^c(0)} K(x) |u_n|^q dx \leq \epsilon C'_q \displaystyle\int_{B_r^c(0)} (V(x) |u_n| + |u_n|^{1^*}) dx , \quad\hbox{for all $n \in \mathbb{N}$}. $$

But since $(u_n)$ is bounded in $E_V$ and also in $L^{1^*}(\mathbb{R}^N)$, it follows that there exists a constant $C'''> 0$ such that

\begin{eqnarray}\label{3.7}
\int_{B_r^c(0)} K(x) |u_n|^q dx \leq C''' \epsilon \quad \forall n \in \mathbb N.
\end{eqnarray}
Now (\ref{eq}) follows from (\ref{3.7}) in the same way as in the first case. This proves the result.
\end{proof}

\begin{proposition}
Assume that $(V,K) \in \mathcal{K}$ and $f$ satisfies $(f_1)-(f_3)$. Let $(u_n)\subset E_V$ be a bounded sequence such that $u_n \to u$ in $L^r_K(\mathbb{R}^N)$ for all $r \in (1,1^*)$ if $(VK_2)$ holds and for $r = q$ if $(VK_3)$ holds. then
$$
\lim_{n \to \infty}\int_{\mathbb{R}^N}F(u_n)dx = \int_{\mathbb{R}^N}F(u)dx
$$
and
$$
\lim_{n \to \infty}\int_{\mathbb{R}^N}f(u_n)u_ndx = \int_{\mathbb{R}^N}f(u)udx.
$$
\label{propositioncompactnessf}
\end{proposition}

The last result can be proved following the same arguments that in \cite{AlvesSouto}[Lemma 2.2].

For later use, let us prove a slight modified version of Lemma \ref{semicontinuity}, which have the convergence in weighted Lebesgue spaces as an assumption.
\begin{lemma}
Let $(u_n) \in E_V$ be a bounded sequence such that $u_n \to u$ in $L_K^{r}(\mathbb{R}^N)$, for some $1 < r < 1^*$. Then it follows that $u \in E_V$ and
$$\int_{\mathbb{R}^N} |Du| \leq \liminf_{n \to \infty} \int_{\mathbb{R}^N} |Du_n|.$$
\label{semicontinuityK}
\end{lemma}
\begin{proof}
First of all note that by Fatou Lemma it follows that $u \in L^{1^*}(\mathbb{R}^N)$. Now, let $\phi \in C^1_c(\mathbb{R}^N,\mathbb{R}^N)$ such that $|\phi|_\infty \leq 1$. By $(VK_0)$, $u_n \to u$ in $L^r(supp(\varphi))$ and as a consequence, in $L^1(supp(\varphi))$. Then
$$\int_{\mathbb{R}^N}u \mbox{div}\phi dx = \int_{supp(\phi)}u \mbox{div}\phi dx = \lim_{n \to +\infty}\int_{supp(\phi)}u_n \mbox{div}\phi dx \leq \liminf_{n \to \infty}\int_{\mathbb{R}^N}|Du_n|.$$
Then the result follows by taking the supremum over all such $\phi$.
\end{proof}

\section{Proof of Theorem \ref{theoremapplication}}

\begin{proof}[Proof of Theorem \ref{theoremapplication}]
First of all let us prove that the restriction of $\Phi$ to the Banach space $E_V$ satisfies condition $i)$ of the Mountain Pass Theorem. But before, just note that by $(f_2)$ and $(f_3)$ it follows that for all $\eta > 0$, there exists $A_\eta > 0$ such that
\begin{equation}
|F(s)| \leq \eta|s| + A_\eta|s|^{1^*}, \quad \forall s \in \mathbb{R}.
\label{fepsilon}
\end{equation}

Note that, by (\ref{fepsilon}), $(VK_0)$ and the fact that $E_V \hookrightarrow L^{1^*}(\mathbb{R}^N)$, it follows that
\begin{eqnarray*}
\Phi(u)& = &\int_{\mathbb{R}^N} |Du| + \int_{\mathbb{R}^N} V(x)|u| dx - \int_{\mathbb{R}^N} K(x)F(u)dx\\
& \geq & \|u\| - C_1\eta  \int_{\mathbb{R}^N} V(x)|u| dx - C_1 A_\eta|u|_{1^*}^{1^*}\\
& \geq & (1-C_1\eta)\|u\| - C_2\|u\|^{1^*}\\
& \geq & \alpha,
\end{eqnarray*}
for all $u \in E_V$, such that $\|u\| = \rho$, where $0 < \eta < 1/{C_1}$ is fixed, $\displaystyle 0 < \rho < \left(\frac{1-C_1\eta}{C_2}\right)^\frac{1}{1^*-1}$ and $\displaystyle \alpha = \rho(1 - C_1\eta - C_2\rho^{1^*-1})$.

Now let us prove that $\Phi$ satisfies the condition $ii)$ of Theorem \ref{mountainpass}.
First note that condition $(f_4)$ implies that there exists constants $d_1,d_2 > 0$ such that
\begin{equation}
F(s) \geq d_1|s|^\theta - d_2 ,\quad \forall s \in \mathbb{R}.
\label{Festimate}
\end{equation}

If $u \in E_V$, $\mbox{supp}(u)$ is compact, $u \neq 0$ and $t > 0$, then
$$
\Phi(tu) \leq t\|u\| - d_1|K|_\infty t^\theta |u|_\theta^\theta + d_2|K|_\infty|\mbox{supp}(u)| \to -\infty,
$$
as $t \to +\infty$, since $\theta > 1$.

Then, Theorem \ref{mountainpass} implies that there exists $(u_n)\subset E_V$ and $\epsilon_n \to 0$ such that 
$$
\Phi_n(u_n) \to c, \quad \mbox{as $n \to +\infty$}
$$
and
\begin{equation}
\mathcal{J}(v) - \mathcal{J}(u_n) \geq \int_{\mathbb{R}^N}K(x)f(u_n)(v - u_n)dx - \epsilon_n\|v - u_n\|, \quad \forall v \in E_V.
\label{mountainpasssequence1}
\end{equation}

Let us prove that the sequence $(u_n)$ is bounded in $E_V$. In (\ref{mountainpasssequence1}), let us take as test function $v = 2u_n$ and note that
$$
\|u_n\| \geq \int_{\mathbb{R}^N}K(x)f(u_n)u_ndx + \epsilon_n\|u_n\|,
$$
which implies that
\begin{equation}
(1 - \epsilon_n)\|u_n\| \geq \int_{\mathbb{R}^N}K(x)f(u_n)u_ndx.
\label{eqn1}
\end{equation}
Then, by $(f_4)$ and (\ref{eqn1}), note that
\begin{eqnarray*}
c + o_n(1) & \geq & \Phi(u_n)\\
& = & \|u_n\| + \int_{\mathbb{R}^N}K(x)\left(\frac{1}{\theta}f(u_n)u_n - F(u_n)\right)dx - \int_{\mathbb{R}^N}K(x)\frac{1}{\theta}f(u_n)u_ndx\\
& \geq & \|u_n\|\left(1- \frac{1}{\theta} + \frac{\epsilon_n}{\theta}\right)\\
& \geq & C\|u_n\|,
\end{eqnarray*}
for some $C > 0$ uniform in $n \in \mathbb{N}$. Then it follows that $(u_n)$ is bounded.

By the boundedness of $(u_n) \subset E_V$ and Proposition \ref{propositioncompactness}, it follows that there exists $u \in E_V$ such that $u_n \to u$ in $L^r_K(\mathbb{R}^N)$ for all $1 < r < 1^*$ if $(VK_2)$ holds and in $L^q(\mathbb{R}^N)$ if $(VK_3)$ holds . By Lemma \ref{semicontinuityK}, it follows that $u \in E_V$.
By the boundedness of $(u_n)$ in $E_V$, Proposition \ref{propositioncompactnessf} and Lemma \ref{semicontinuityK} it follows calculating the $\liminf$ both sides of (\ref{mountainpasssequence1}) that
\begin{equation}
\mathcal{J}(v) - \mathcal{J}(u) \geq \int_{\mathbb{R}^N} K(x)f(u)(v-u), \quad \forall v \in E_V.
\label{BVsolution1}
\end{equation}
Hence, $u$ is a bounded variation solution of (\ref{Pintro}).

Moreover, by taking $v = u + tu$ for $t > 0$ in (\ref{BVsolution1}), note that
$$
\frac{\mathcal{J}(u + tu) - \mathcal{J}(u)}{t} \geq \int_{\mathbb{R}^N}K(x)f(u)udx.
$$
Passing the limit as $t \to 0^+$, it follows that
$$
\mathcal{J}'(u)u \geq \int_{\mathbb{R}^N}K(x)f(u)udx.
$$
By doing the same for $t < 0$ and passing the limit as $t \to 0^-$, it follows that
$$
\mathcal{J}'(u)u \leq \int_{\mathbb{R}^N}K(x)f(u)udx.
$$
Then, by the last inequalities and (\ref{derivadaJ}),
\begin{equation}
\mathcal{J}(u) = \mathcal{J}'(u)u = \int_{\mathbb{R}^N}K(x)f(u)udx.
\label{A1}
\end{equation}

Applying the same arguments with $v = u_n + tu_n$ in (\ref{mountainpasssequence1}), for $t > 0$ and $t < 0$ and by making $t \to 0^+$ and $t \to 0^-$, it follows that
\begin{equation}
\mathcal{J}(u_n) = \mathcal{J}'(u_n)u_n = \int_{\mathbb{R}^N}K(x)f(u_n)u_ndx + o_n(1).
\label{A2}
\end{equation}
Then, from (\ref{A1}), (\ref{A2}) and Proposition \ref{propositioncompactnessf}, it follows that
\begin{eqnarray*}
\mathcal{J}(u) & = & \int_{\mathbb{R}^N}K(x)f(u)udx\\
& = & \lim_{n \to \infty}\int_{\mathbb{R}^N}K(x)f(u_n)u_ndx\\
& = & \lim_{n \to \infty} \mathcal{J}(u_n).
\end{eqnarray*}

Finally, the last equality and Proposition \ref{propositioncompactnessf} imply that
$$c = \lim_{n \to \infty}\Phi(u_n) = \Phi(u)$$
and then $u \neq 0$.

Now, what is left to justify is just that the solution $u$ in fact is a ground-state solution, i.e., that $u$ has the lowest energy level among all nontrivial bounded variation solutions. In order to prove it, we have to recall \cite{FigueiredoPimenta}, where is proved that we can define the Nehari set associated to $\Phi$, given by
$$
\mathcal{N} = \left\{u \in E_V\backslash\{0\}; \, \int_{\mathbb{R}^N}|Du| + \int_{\mathbb{R}^N}V(x)|u|dx = \int_{\mathbb{R}^N}f(u)udx   \right\}.
$$
It is proven in \cite{FigueiredoPimenta} that $\mathcal{N}$ is a set which contains all nontrivial bounded variation solutions of (\ref{Pintro}). Then, if we manage to prove that the solution $u$ is such that $\Phi(u) = \inf_{\mathcal{N}}\Phi$, then $u$ would be a ground-state solution of (\ref{Pintro}).

By using the same kind of arguments that Rabinowitz in \cite{Rabinowitz}, in the light of $(f_1) - (f_5)$, one can easily see that $\mathcal{N}$ is radially homeomorphic to the unit sphere in $E_V$ and also that the minimax level $c$ satisfies
$$
c = \inf_{v \in E_V\backslash\{0\}}\max_{t \geq 0} \Phi(tv) = \inf_{v \in \mathcal{N}}\Phi(v).
$$
Since the solution $u$ is such that $\Phi(u) = c$, it follows that $u$ is a ground-state bounded variation solution and this concludes the proof of Theorem \ref{theoremapplication}.

The fact that $u$ satisfies the Euler-Lagrange equation (\ref{eulerlagrangeintro}) just follows from the Subsection \ref{eulerlagrangesection}.

\end{proof}

\section{Appendix}

\hspace{.5cm} In this appendix we state and prove some important results we have used in the last sections.

\begin{theorem}
Let $E$ be a Banach space, $\Phi = I_0 - I$ where $I \in C^1(E,\mathbb{R})$ and $I_0$ a locally Lipschitz convex functional defined in $E$.
Suppose that the functional $\Phi$ satisfies:
\begin{itemize}
\item [$i)$] There exist $\rho > 0$, $\alpha > \Phi(0)$ such that $\displaystyle \Phi|_{\partial B_\rho(0)} \geq \alpha$,
\item [$ii)$] $\Phi(e) < \Phi(0)$ for some $e \in E \backslash \overline{B_\rho(0)}$,
\end{itemize}
Then for all $\epsilon > 0$ there exists $x_\epsilon \in E$ such that 
\begin{equation}
c - \epsilon < \Phi(x_\epsilon) < c+\epsilon,
\label{MPintro}
\end{equation}
where $c \geq \alpha$ is characterized by
\begin{equation}
c = \inf_{\gamma \in \Gamma} \sup_{t \in [0,1]}\Phi(\gamma(t)),
\label{minimax}
\end{equation}
where $\Gamma = \{\gamma \in C^0([0,1],E); \, \gamma(0) = 0 \, \,  \mbox{and} \, \,  \gamma(0) = e\}$
and
\begin{equation}
I_0(y) - I_0(x_\epsilon) \geq I'(x_\epsilon)(y - x_\epsilon) - \epsilon \|y-x_\epsilon\|, \quad \forall y \in E.
\label{MP2}
\end{equation}
\label{mountainpass}
\end{theorem}

Before we start proving Theorem \ref{mountainpass}, let us prove that condition (\ref{MP2}) is equivalent to the existence of $z_\epsilon \in E^*$ such that $\|z_\epsilon\|_* \leq \epsilon$ and
\begin{equation}
I_0(y) - I_0(x_\epsilon) \geq I'(x_\epsilon)(y - x_\epsilon) + \langle z_\epsilon,y-x_\epsilon\rangle_{E^*,E}, \quad \forall y \in E,
\label{MP3}
\end{equation}
where $\langle \cdot,\cdot \rangle_{E^*,E}$ denotes the duality pair between $E$ and its dual.

In fact, clearly (\ref{MP3}) implies (\ref{MP2}).
In order to prove that (\ref{MP2}) also imply (\ref{MP3}), let us state a lemma proved by Szulkin in \cite{Szulkin}[Lemma 1.3].

\begin{lemma}
Let $E$ be a Banach space and $\chi:E \to (-\infty,+\infty]$ a lower semicontinuous convex function such that $\chi(0) = 0$. If 
$$\chi(x) \geq -\|x\|, \quad \forall x \in E,$$
then there exists $z \in E^*$, $\|z\|_* \leq 1$, such that
$$\chi(x) \geq \langle z,x \rangle_{E^*,E}, \quad \forall x \in E.$$
\label{lemmaszulkin}
\end{lemma}

Now, if (\ref{MP2}) holds, then
$$\frac{1}{\epsilon}\left(I_0((y-x_\epsilon)+x_\epsilon) - I_0(x_\epsilon) - I'(x_\epsilon)(y - x_\epsilon)\right) \geq -\|y-x_\epsilon\|,$$
for all $y \in E$. By applying Lemma \ref{lemmaszulkin} to
$$\chi(x) = \frac{1}{\epsilon}\left(I_0((x + x_\epsilon) - I_0(x_\epsilon) - I'(x_\epsilon)x\right)$$
it follows that there exists $z \in E^*$, such that $\|z\|_* \leq 1$ and
$$\chi(x) \geq \langle z,x \rangle_* \quad \forall x \in E.$$
Taking $z_\epsilon = \epsilon z$ and $x = y-x_\epsilon$ where $y \in E$, it follows (\ref{MP3}) for $z_\epsilon$ and $\|z_\epsilon\|_* \leq \epsilon$.

To proceed with the proof, we need a version of Deformation Lemma without the Palais-Smale condition which has been proved in \cite{FigueiredoPimenta}[Theorem 4]. By the sake of completeness we state and prove it again here.
\begin{theorem}[Deformation lemma]

Let $E$ be a Banach space and $T:E \to \mathbb{R}$ a locally Lipschitz functional. When $a \in \mathbb{R}$, let us denote $T_a = \{x \in E; \, T(x) \leq a\}$. If there exist $d \in \mathbb{R}$, $S \subset E$ and $\alpha, \delta, \epsilon_0 > 0$ such that 
$$\beta(x):= \min\{\|z\|_{E^*}; \, z \in \partial T(x)\} \geq \alpha, \quad \forall x \in T^{-1}([d-\epsilon_0,d+\epsilon_0]) \cap S_{2\delta},$$
where $S_{2\delta}$ is a $2\delta-$neighborhood of $S$, then for $0 < \epsilon < \min\left\{\frac{\delta\alpha}{2},\epsilon_0\right\}$, there exists an homeomorphism $\eta:E \to E$ such that
\begin{itemize}
\item [$i)$] $\eta(x) = x$ for all $x \not \in  T^{-1}([d-\epsilon_0,d+\epsilon_0])\cap S_{2\delta}$;
\item [$ii)$] $\eta(T_{d+\epsilon}\cap S) \subset T_{d-\epsilon}$;
\item [$iii)$] $T(\eta(x)) \leq T(x)$, for all $x \in E$.
\end{itemize}
\label{deformationlemma}
\end{theorem}

\begin{proof}[Proof of Theorem \ref{deformationlemma}]
To start with, under these assumptions, let us recall Lemma 3.3 of \cite{Chang}, which states the existence of a {\it psudo-gradient vector field} for $T$, given by a locally Lipschitz vector field $g: T^{-1}([d-\epsilon_0,d+\epsilon_0])\cap S_{2\delta} \to E$ satisfying
\begin{equation}
\|g(x)\| < 1
\label{vectorfield1}
\end{equation}
and 
\begin{equation}
\langle z^*,g(x)\rangle_{E^*,E} > \frac{\alpha}{2}, \quad \forall z^* \in \partial T(x).
\label{vectorfield2}
\end{equation}

For
\begin{equation}
0 < \epsilon < \min\left\{\frac{\delta\alpha}{2},\epsilon_0\right\},
\label{epsilon}
\end{equation}
define
$$A = T^{-1}([d-\epsilon_0,d+\epsilon_0]) \cap S_{2\delta},$$
$$B = T^{-1}([d-\epsilon,d+\epsilon]) \cap S_\delta$$
and note that $B \subset A$. Define 
$$\psi(x) = \frac{d(x,E\backslash A)}{d(x,E\backslash A) + d(x,B)}$$
and note that $\psi$ is a locally Lipschitz continuous function such that $0 \leq \psi \leq 1$ and
$$
\psi(x) = \left\{
\begin{array}{ll}
1 & \mbox{if $x \in B$,}\\
0 & \mbox{if $x \in E\backslash A$.}\\
\end{array} \right.
$$
Now consider $V(x) = \psi(x)g(x)$ which is also locally Lipschitz continuous and $\sigma(t,x)$ the solution of
$$
\left\{
\begin{array}{rl}
\displaystyle \frac{d}{dt}\sigma(t,x) & = -V(\sigma(t,x)), \quad t > 0,\\
\sigma(0,x) & = x,
\end{array}\right.
$$
which is continuous in $\mathbb{R}_+ \times E$.

Let us choose
\begin{equation}
t_0 \in \left(\frac{2\epsilon}{\alpha}, \delta\right)
\label{t_0}
\end{equation}
and define 
$$
\eta(x) = \sigma(t_0,x), \quad x \in E.
$$

Note that since $V \equiv 0$ in $E \backslash (T^{-1}([d-\epsilon_0,d+\epsilon_0]) \cap S_{2\delta}$, it follows that $i)$ holds.

To prove $ii)$, let us first recall Proposition 9 in \cite{Chang} which implies that $t \mapsto T(\sigma(t,x))$ is a.e. differentiable, for each $x \in E$. Moreover, we have that
\begin{equation}
\begin{array}{lll}
\displaystyle \frac{d}{dt}T(\sigma(t,x)) & \leq & \max\left\{\left< z^*,\frac{d}{dt}\sigma(t,x)\right>_{E^*,E}; \, z^* \in \partial T(\sigma(x,t))\right\}\\
& = & \displaystyle  -\min\{\left< z^*,V(\sigma(t,x))\right>; \, z^* \in \partial T(\sigma(x,t))\}\\
& \leq & \displaystyle  \left\{
\begin{array}{rl}
\displaystyle  - \frac{\alpha}{2} & \mbox{if $\sigma(t,x)\subset T^{-1}([d-\epsilon,d+\epsilon])\cap S_\delta$}\\
0 & \mbox{otherwise,}
\end{array}\right.
\end{array}
\label{estimativadeformacao}
\end{equation}
where we use (\ref{vectorfield2}) in the last estimate. Then the function $t \mapsto T(\sigma(t,x))$ is nonincreasing, for all $x \in E$ and then we get $iii)$.

Note also that, for all $t > 0$
\begin{eqnarray*}
\|\sigma(t,x) - x\| & = & \|\sigma(t,x) - \sigma(0,x)\\
& = & \left\| \int_0^t \frac{d}{ds}\sigma(s,x)ds\right\|\\
& \leq & \int_0^t \|V(\sigma(s,x))\|ds\\
& \leq & t.
\end{eqnarray*}

Let us take $x \in T_{d+\epsilon} \cap S$. If there exists some $t \in \left[0,t_0\right]$ such that $T(\sigma(t,x)) < d-\epsilon$, then $T(\sigma(t_0,x)) < d-\epsilon$ and $ii)$ is satisfied by $\eta$. Then suppose that
$$\sigma(t,x) \in T^{-1}([d-\epsilon,d+\epsilon]), \forall t \in [0,t_0]$$
and let us prove that $\sigma(t,x) \subset S_\delta$, $\forall t \in [0,t_0]$. In fact, note that
$$\|\sigma(t,x) - x\| \leq t \leq t_0 < \delta, \quad \forall t \in [0,t_0].$$

Hence, since $\sigma([0,t_0],x) \subset T^{-1}([d-\epsilon,d+\epsilon]) \cap S_\delta$, it follows by (\ref{t_0}) and (\ref{estimativadeformacao}) that

\begin{eqnarray*}
T(\eta(x))& = & T(\sigma(t_0,x))\\
& = & T(x) + \int_0^{t_0} \frac{d}{dt}T(\sigma(s,x))dx\\
& \leq & T(x) - \frac{\alpha}{2}t_0\\
& < & d -\epsilon
\end{eqnarray*}
and $ii)$ follows.
\end{proof}

Now, finally, let us proceed with the proof of Theorem \ref{mountainpass}.
\begin{proof}[Proof of Theorem \ref{mountainpass}]
First, note that since $\Phi(e) < \Phi(0) < \alpha \leq \Phi\left|_{\partial B_\rho}\right.$, then 
$$c \geq \alpha.$$

Suppose by contradiction that there exists $\epsilon > 0$, which can be assumed to satisfy 
$$c - \epsilon > \Phi(0),$$
such that for all $x \in \Phi^{-1}([c-\epsilon,c+\epsilon])$ where $c$ is defined in (\ref{minimax}), (\ref{MP2}) does not hold. Since (\ref{MP2}) is equivalent to (\ref{MP3}), this implies that for all $z \in E^*$ such that $\|z\|_* \leq \epsilon$, there exists $y_\epsilon \in X$, such that
$$
I_0(y_\epsilon) - I_0(x_\epsilon) < I'(x_\epsilon)(y_\epsilon - x_\epsilon)dx + \langle z_\epsilon,y_\epsilon-x_\epsilon\rangle_*.
$$
Hence, it follows that 
$$\beta(x) \geq \epsilon, \quad \forall x \in \Phi^{-1}([c-\epsilon,c+\epsilon]),$$
where $\beta(x) = \inf\{\|w^*\|_*; \, \, w^* \in \partial I_0(u) - I'(x)\}$.

By Theorem \ref{deformationlemma} applied to $T = \Phi$, $d=c$, $\alpha = \epsilon$ and $\epsilon_0 = \epsilon$ it follows that there exists an homeomorphism $\eta: E \to E$ and $\bar{\epsilon} \in (0,\epsilon)$ such that 
\begin{itemize}
\item [$i)$] $\eta(x) = x$ for all $x \not \in  \Phi^{-1}([c-\epsilon,c+\epsilon])$;
\item [$ii)$] $\eta(\Phi_{c+\bar{\epsilon}}) \subset \Phi_{c-\bar{\epsilon}}$.
\end{itemize}

By the definition of $c$, there exists $\gamma \in \Gamma$ such that
$$c \leq \max_{t \in [0,1]}\Phi(\gamma(t)) \leq c+\overline\epsilon.$$

Let us consider $\tilde{\gamma}(t) = \eta(\gamma(t))$ and note that, since $\Phi(0), \Phi(e) < c-\epsilon$, $i)$ implies that $\tilde{\gamma} \in \Gamma$.
Then, $ii)$ implies that
$$c \leq \max_{t\in[0,1]}\Phi(\tilde{\gamma}(t)) \leq c - \overline{\epsilon},$$
which is a contradiction. Then the result follows.

\end{proof}

\noindent \textbf{Acknowledgment.}\  

Giovany Figueiredo and Marcos Pimenta would like to thank FAPESP and CNPq for the financial support. This work was written while Giovany M. Figueiredo was as a Visiting Professor at FCT - Unesp in Presidente Prudente. He would like to thanks the warm hospitality.

\end{document}